\documentclass[11pt,english]{article}
\usepackage[T1]{fontenc}
\usepackage[latin9]{inputenc}
\usepackage{babel}
\usepackage{amsmath, color}
\usepackage{amsthm}
\usepackage{amssymb}
\usepackage[unicode=true,
 bookmarks=false,
 breaklinks=false,pdfborder={0 0 1},backref=section,colorlinks=false]
 {hyperref}
\usepackage{breakurl}

\makeatletter
\theoremstyle{plain}
\newtheorem{thm}{\protect\theoremname}
  \theoremstyle{plain}
  \newtheorem{lem}[thm]{\protect\lemmaname}
  \theoremstyle{remark}
  \newtheorem{rem}[thm]{\protect\remarkname}
  \theoremstyle{plain}
  \newtheorem{prop}[thm]{\protect\propositionname}
  \theoremstyle{definition}
  \newtheorem{example}[thm]{\protect\examplename}

\usepackage{mathrsfs}\usepackage{amscd}

\usepackage{graphics}
\usepackage{a4wide}

\usepackage{accents}

\usepackage{array}

\newcommand{\R}{\mathbb{R}}

\newcommand{\Z}{\mathbb{Z}}

\newcounter{EQNR}
\renewcommand{\contentsname}{Contents\\
{\footnotesize\normalfont(A table of contents should normally not
be included)}}

\usepackage{babel}
\providecommand{\lemmaname}{Lemma}
\providecommand{\theoremname}{Theorem}
\providecommand{\propositionname}{Proposition}
\providecommand{\remarkname}{Remark}
\providecommand{\examplename}{Example}

\makeatother

  \providecommand{\examplename}{Example}
  \providecommand{\lemmaname}{Lemma}
  \providecommand{\propositionname}{Proposition}
  \providecommand{\remarkname}{Remark}
\providecommand{\theoremname}{Theorem}

\begin{document}

\title{The parametrix construction of the heat kernel on a graph}


\author{Gautam Chinta \and Jay Jorgenson \footnote{The second-named author acknowledges grant support
from PSC-CUNY Award 65400-00-53, which was jointly funded
by the Professional Staff Congress and The City University of New York.} \and Anders Karlsson
\footnote{The third-named author acknowledges grant support by the Swiss NSF grants 200020-200400, 200021-212864 and the Swedish
Research Council grant 104651320.} \and Lejla Smajlovi\'{c}}
\maketitle
\begin{abstract}\noindent
In this paper we develop the parametrix approach for constructing the heat kernel
on a graph $G$.  In particular, we highlight two specific cases.  First, we consider
the case when $G$ is embedded in a Eulidean domain or manifold $\Omega$, and
we use a heat kernel associated to $\Omega$ to obtain a formula for the heat kernel
on $G$.  Second, we consider when $G$ is a subgraph of a larger graph $\widetilde{G}$,
and we obtain a formula for the heat kernel on $G$ from the heat kernel on $\widetilde{G}$
restricted to $G$.
\end{abstract}

\section{Introduction}

A heat kernel is the fundamental solution to a heat equation in a
domain, manifold or graph, which may be subject to certain boundary conditions.
Heat kernels have proved useful when studying a variety
of problems, from geometry and topology to combinatorics and mathematical
physics; see, for example \cite{Gr09} or \cite{JoLa01} for some general discussion
as to the many manifestations of heat kernel techniques throughout mathematics.
Algebraically, the heat kernel can be viewed as a generating function of
the spectrum of the Laplace operator, in instances when its spectrum is discrete.

In this article we will prove a construction of the heat kernel on graphs analogous
to the parametrix construction in Riemannian geometry, \cite{MP49, Ch84, Ro97}. The parametrix approach plays an important role in geometry and in  varied applications such as machine learning \cite{LL05} and quantum field theory \cite{Av01}.  In brief, this approach to constructing the heat kernel on a Riemannian manifold $M$ begins with an initial approximation which could, for example, come from the Euclidean heat kernel patched together on local charts. This is then refined through an iterative process.  The successive refinements provide an increasingly precise small time asymptotic expansion of the heat kernel.

We will describe the construction of the heat kernel on a finite graph by starting with a general parametrix.  We pay particular attention to situations in which a finite graph $G$ is embedded in an ambient space, which may be a domain in $\R^{n}$,
a manifold or another larger graph.  Then we may take our initial approximation in the parametrix construction to be the heat kernel on the ambient space. This allows us to extend our results to construction of the heat kernel on infinite subgraphs with a finite boundary, see Example \ref{ex: discrete half line} below in which we explicitly compute the heat kernel on a half-line. A similar construction in the special case when $G$ is embedded
in a complete graph is considered in \cite{LNY21}.  In Section \ref{sec: heat kernels on subgraphs} below we relate our method to
the approach in \cite{LNY21} and show how our construction implies one of the main results of \cite{LNY21}.

Let us now be more specific in describing the results of this article.
Let $G$ be a finite, edge-weighted graph with finite vertex set
$VG$ and nonnegative weight function $w:VG\times VG\to [0,\infty), w\mapsto w_{xy}$.  When $w_{xy}=0$ we say
there is no edge between $x$ and $y$. Furthermore, we assume the
symmetry $w_{xy}=w_{yx}$ holds for all $x,y\in VG$.  This data defines
the Laplace operator $\Delta_G$ on functions $f:VG\rightarrow\mathbb{C}$
by
\[
\Delta_{G}f(x)=\sum_{y\in VG}(f(x)-f(y))w_{xy}.
\]
Note that the sum in fact is taken over the neighbors of $x$, that is, over the vertices $y$ with weight $w_{xy}>0$.
The \emph{heat kernel} on the
graph $G$ associated to the weighted graph Laplacian $\Delta_{G,v_{1}}$
acting on functions in the first variable $v_{1}$, is the unique
solution $H_{G}(v_{1},v_{2};t)$
to the differential equation
\[
\left(\Delta_{G,v_{1}}+\frac{\partial}{\partial t}\right)H_{G}(v_{1},v_{2};t)=0
\]
with the property that
\[
\lim_{t\to0}H_{G}(v_{1},v_{2};t)=\left\{ \begin{array}{c}
1\textrm{ if \ensuremath{v_{1}=v_{2}}}\\
0\textrm{ if \ensuremath{v_{1}\neq v_{2}}.}
\end{array}\right.
\]
For the proof of the existence and uniqueness of the heat kernel on
$G$, we refer to \cite{Do84} and \cite{DM06}.  We do not normalize the Laplacian $\Delta_{G}f(x)$ by dividing the sum defining it by the weighted degree $\mu(x)=\sum_{y\in VG}w_{xy}$ of $x$; this normalization appears mostly when discussing stochastic properties of graphs, see e.g. \cite{Wo21} and the extensive bibliography there. However, the parametrix construction of the heat kernel given in our main result Theorem \ref{thm: formula for heat kernel} can be carried out in the same manner for the ``normalized'' Laplacian.

The aim of this paper is to describe a parametrix approach to the
heat kernel on $G$, inspired by the work of Minakshisundaram and
Minakshisundaram and Pleijel, see  \cite{Mi49}, \cite{Mi53}
and \cite{MP49} as well as \cite{Ch84}.
After we give a reasonably general construction, we will highlight
the following two special cases.
\begin{enumerate}
\item When the set $VG$ is a subset of a bounded domain $\Omega$ in $\mathbb{R}^{N}$,
we show that the parametrix of an arbitrary order for the heat kernel
on $G$ can be obtained from the heat kernel on $\Omega$ and
a certain partition of unity depending upon a (Dirichlet-)Vorono\u{\i}
decomposition of $\Omega$ induced by $G$.
\item When $G$ is a finite or infinite subgraph of a
graph $\tilde{G}$, we prove that the heat kernel on $\tilde{G}$
can be used as a parametrix of order zero in the construction of the
heat kernel on $G$, under the technical assumption that the boundary $\partial G$ of $G$ when viewed as a subgraph of $\tilde{G}$ has a finite
number of vertices.
\end{enumerate}

Any new expression for the heat kernel is potentially a source for interesting new identities, particularly in the settings where the heat kernel is unique.
For example, when
studying continuous functions on the continuous circle, one can write the heat kernel
either through its spectral expansion or through the periodization of the
heat kernel on the covering space.  In doing so, one immediately obtains a
classical theta inversion formula, which itself is logically equivalent
to the one-variable Poisson summation formula.  A similar argument in the setting
of compact hyperbolic Riemann surfaces yields a quick proof of the Selberg trace formula;
 see Remark 3.3 of \cite{GJ18}.
In the case of a heat kernel on a discrete circle
with $N$ points, then similar considerations in \cite{KN06} lead to ``heat kernel proofs'' of $I$-Bessel identities. See also \cite{Gr22} and \cite{CHJSV23} for proofs of trigonometric identities using the uniqueness of the discrete-time heat kernel on a discrete circle with $N$ points.

With this in mind, in Example \ref{ex: half line dirichhlet heat} we specialize our general result to the discrete half-line $\Z_{\geq 0}$
with a Dirichlet boundary condition.
Let $I_{x}(t)$ be the $I$-Bessel function and define the convolution
\begin{equation}\label{eq:one_variable_conv}
\left( I_{x}\ast I_{y}\right)(t) = \int\limits_{0}^{t}I_{x}(\tau)I_{y}(t-\tau)d\tau.
\end{equation}
Then the parametrix construction of the heat kernel on $\Z_{\geq 0}$ yields the identity
\begin{equation}\label{eq:I-Bessel_DLine}
I_{x+y}(t)=\sum_{\ell=0}^{\infty}\frac{(-1)^{\ell}}{2^{\ell+1}}\left(I_{1}^{\ast\ell}
\ast I_{x-1}\ast I_{y}\right)(t)
\end{equation}
where $I_{1}^{\ast\ell}$ denotes the $\ell$-fold convolution of $I_{1}$ itself. This identity can be confirmed directly by comparing the Laplace transform of both sides.

The paper is organized as follows. In Section \ref{sec: conv on a graph} we
develop various analytic results which are needed in our study. Section
\ref{sec:The-parametrix-construction} contains the details of the general
parametrix construction. In Section \ref{sec:Graphs-embedded-in} we relate the
heat kernel of a graph embedded in a domain with the heat kernel of the domain. Remark
\ref{rem:Domain_to_graph} is particularly interesting since it gives a precise
formula for the (graph) heat kernel on a graph $G$ which is embedded in a domain
$\Omega$ in terms of the (geometric) heat kernel associated to $\Omega$; see
equation (\ref{eq:heat kernel with k=00003D0}). Section
\ref{sec:Heat-kernels-subgraphs} deals with derivation of expressions for the
heat kernel in a general subgraph case. In Section \ref{sec: heat kernels on subgraphs}
we show that our general result, applied to the setting of the
complete graph on $N$ vertices yields some of the main results from Section 3 of \cite{LNY21}.  We
also give an alternative formulation for those results and explicitly compute
the heat kernel on the graph obtained by removing a single edge from the
complete graph. An example of an infinite subgraph is treated in Section
\ref{sec: half-line}. In Section \ref{sec:Dirichlet-heat-kernel}, we show how to
construct the Dirichlet heat kernel on a subgraph $G$ of a finite or infinite
graph, assuming that boundaries of $G$ and $G\setminus\partial G$ are finite. An
example in this section yields the $I$-Bessel function identity
(\ref{eq:I-Bessel_DLine}).


\section{Convolution on a graph     
\label{sec: conv on a graph}}       

Let $G$ be a finite weighted graph as above with vertex set $VG$ of cardinality $|VG|$. Let
$F_{1},F_{2}:VG\times VG\times\mathbb{R}_{>0}\to\mathbb{R}$. Assume further that
for $i=1,2$ and fixed vertices $v_{1}, v_{2}$, the functions $F_{i}(v_{1},v_{2};t)$ are
integrable as functions of $t$ on every interval $(0,b]$ with $b>0$. The \emph{convolution} of functions $F_{1}$ and $F_{2}$ is
defined to be
\begin{equation}\label{def: convolution}
(F_{1}\ast F_{2})(v_{1},v_{2};t):=\int\limits _{0}^{t}\sum_{v\in VG}F_{1}(v_{1},v,t-r)F_{2}(v,v_{2};r)dr.
\end{equation}
When thinking of $F_{1}$ and $F_{2}$ as operators on $L^{2}(VG)$, or equivalently as $|VG|\times |VG|$ matrices, we may simply write
\begin{equation}
(F_{1}\ast F_{2})(t):=\int\limits _{0}^{t}F_{1}(t-r)\circ F_{2}(r)dr.
\end{equation}

The above convolution is not commutative in general but it is associative.  Specifically,
for any three functions $F_{i}(v_{1},v_{2};t)$,
$i=1,2,3$, $F_{i}:VG\times VG\times\mathbb{R}_{>0}\to\mathbb{R}$
which are integrable as functions of $t$ on every interval $(0,b]$,
$b>0$ for all $v_{1},v_{2}\in VG$, we have
\begin{equation}\label{eq:two_variable_conv}
[(F_{1}\ast F_{2})\ast F_{3}](v_{1},v_{2};t)=[F_{1}\ast(F_{2}\ast F_{3})](v_{1},v_{2};t).
\end{equation}

\begin{rem} There are two notions of convolution which are used in this paper: the two-variable vertex and time convolution \eqref{def: convolution} defined in this section
and a one-variable time convolution as in \eqref{eq:one_variable_conv}.  The latter is used only in Example \ref{ex: half line dirichhlet heat} of Section \ref{sec:Dirichlet-heat-kernel} and we hope no confusion will arise from using the same symbol $\ast$ to denote both types of convolution. \end{rem}

We have the following lemma.
\begin{lem}
\label{lem:convolution bounds} Let $F_{1},F_{2}:VG\times VG\times\mathbb{R}_{>0}\to\mathbb{R}$
be as above.  For some $t_{0} > 0$, assume there exist constants $C_1, C_2$ and integers $k,\ell\geq 0$
such that for all $0<t<t_{0}$ and $v_1,v_2 \in VG$, we have
$$
\left|F_{1}(v_{1},v_{2};t)\right|\leq C_{1}t^{k}
  \,\,\,\,\,
  \text{\rm and}
  \,\,\,\,\,
  \left|F_{2}(v_{1},v_{2};t)\right|\leq C_{2}t^{\ell}.
$$
Then, for all $v_{1},v_{2}\in VG$
we have
\[
  |(F_{1}\ast F_{2})(v_{1},v_{2};t)|\leq C_{1}C_{2}|VG|\frac{k!\ell!}{(k+\ell+1)!}
  t^{k+\ell+1}
  \,\,\,\,\,
  \text{\rm for $0<t<t_{0}$.}
\]
In particular, if $k=0$ then
$$
|(F_1\ast F_2)(v_1,v_2;t)|\leq C_1C_2|VG|\frac{t^{\ell+1}}{\ell+1}
\,\,\,\,\,
\text{\rm for all $v_{1},v_{2}\in VG$ and $0<t<t_{0}$.}
$$
\end{lem}

\begin{proof}
The proof is immediate from the definition of the convolution and the
formula for the  Euler beta function, namely that
$$
\int_0^t r^k (t-r)^\ell\, dr=\frac{k!\ell!t^{k+\ell+1}}{(k+\ell+1)!}.
$$
\end{proof}

For any positive integer $\ell$ and any function $f=f(v_{1},v_{2};t):VG\times VG\times\mathbb{R}_{>0}\to\mathbb{R}$
which is integrable in $t$ on every interval $(0,b]$, $b>0$ for
all $v_{1},v_{2}\in VG$, we can inductively define the $\ell$-fold
convolution $(f)^{\ast\ell}(v_{1},v_{2};t)$ by setting $(f)^{\ast1}(v_{1},v_{2};t)=f(v_{1},v_{2};t)$
and, for $\ell\geq2$ we put
$$
(f)^{\ast\ell}(v_{1},v_{2};t):=\left((f)^{\ast(\ell-1)}\ast f\right)(v_{1},v_{2};t).
$$
Note that in the last definition the order is not important because
the convolution operator is associative.

With this notation we have the following lemma.

\begin{lem}
\label{lem:convergence of the series} Let $f=f(v_{1},v_{2};t):VG\times VG\times\mathbb{R}_{>0}\to\mathbb{R}$.
Assume that for all $v_{1},v_{2}\in VG$ and $t_{0} \in \mathbb{R}_{>0}$, the function $f(v_{1},v_{2};\cdot)$ is integral on
the interval $(0,t_{0}]$.
Assume further there exists a constant $C$ and integer $k\geq 0$ such that
$$
f(v_{1},v_{2};t)\leq Ct^{k}
\,\,\,\,\,
\text{\rm for all $v_{1},v_{2}\in VG$ and $0<t<t_{0}$.}
$$
Then the series
\begin{equation}\label{eq: series over ell}
\sum_{\ell=1}^{\infty}(-1)^{\ell}(f)^{\ast\ell}(v_{1},v_{2};t)
\end{equation}
converges absolutely and uniformly on every compact subset of $VG\times VG\times\mathbb{R}_{>0}$.
In addition, we have that
\begin{equation}
\left(f\ast\sum_{\ell=1}^{\infty}(-1)^{\ell}(f)^{\ast\ell}\right)(v_{1},v_{2};t)=
\sum_{\ell=1}^{\infty}(-1)^{\ell}(f)^{\ast(\ell+1)}(v_{1},v_{2};t).\label{eq: convolution with inf sum}
\end{equation}
and,
\begin{equation}\label{eq:series_bound}
\sum_{\ell=1}^{\infty}\left|(f)^{\ast\ell}(v_{1},v_{2};t)\right|
=O(t^{k})
\,\,\,\,\,
\text{\rm as $t\to 0$.}
\end{equation}
\end{lem}

\begin{proof}
With the stated assumptions, together with Lemma \ref{lem:convolution bounds}, we have that
\begin{equation}
  \label{eq:4}
\left|(f\ast f)(v_1,v_2;t)\right| \leq C^2|VG|(k!)^2\frac{t^{2k+1}}{(2k+1)!}.
\end{equation}
Similarly, by induction for $\ell\geq 1$ we have the bound that
\begin{equation}
  \label{eq:mult_convolution_bound}
\left|(f^{\ast(\ell)})(v_1,v_2;t)\right|
\leq
(Ck!)^\ell|VG|^{\ell-1}\frac{t^{\ell k+\ell-1}}{(\ell k+\ell-1)!}.
\end{equation}
The assertion regarding the convergence of (\ref{eq: series over ell}) now follows from the Weierstrass criterion.

Fix $t>0$. By reasoning
as above, it is easy to deduce that the series on the right-hand side of \eqref{eq: convolution with inf sum} converges absolutely. Therefore, we may integrate the equation
$$
\left(f\ast \sum_{\ell=1}^{\infty} (-1)^\ell (f)^{\ast \ell}\right)(v_1,v_2;t)=\int\limits_0^t \sum_{v\in VG} \left( f(v_1,v,t-r) \sum_{\ell=1}^{\infty} (-1)^\ell (f)^{\ast \ell}(v,v_2;r) \right) dr
$$
termwise to establish \eqref{eq: convolution with inf sum}.  Finally, the assertion in
\eqref{eq:series_bound} follows by combining \eqref{eq: series over ell},
\eqref{eq:mult_convolution_bound},  and Lemma \ref{lem:convolution bounds}.

\end{proof}

\section{The parametrix construction of the heat         
  kernel on $G$\label{sec:The-parametrix-construction}}  

We repeat and expand on some notation from the introduction.
Let $G$ be a finite weighted graph with vertex set $VG$ and edge weight function $w_{xy}$ for vertices $x,y.$
The Laplace operator $\Delta_G$ acting on functions $f:VG\rightarrow\mathbb{C}$
is defined by
\begin{equation}\label{eq:Laplacian}
\Delta_{G}f(x)=\sum_{y\in VG}(f(x)-f(y))w_{xy}.
\end{equation}
and the heat operator $L$ on the graph is defined by
\begin{equation}
  \label{def:heat_operator}
  L=\Delta_{G}+\frac{\partial}{\partial t}.
\end{equation}
The heat kernel $H_{G}$ on $G$ associated to the Laplacian $\Delta_G$ is the unique solution
$H_G:VG\times VG\times [0,\infty)$ to the differential equation
\[
L_{v_1}H_{G}(v_{1},v_{2};t)=0
\]
with the property that
$$
H_{G}(v_{1},v_{2};0)=\left\{ \begin{array}{c}
1\,\,\,\textrm{ if \ensuremath{v_{1}=v_{2}}}\\
0\,\,\,\textrm{ if \ensuremath{v_{1}\neq v_{2}}.}
\end{array}\right.
$$
The subscript $v_{1}$ on $L_{v_1}$ indicates the sum associated to the Laplacian, as in \eqref{eq:Laplacian} is
over neighbors of $v_1$ which is  the first space variable.
A \emph{parametrix} for the heat operator on $G$ is any continuous function $H:VG\times VG\times [0,\infty)$
which is smooth on $VG\times VG\times (0,\infty)$ and satisfies the following two properties.
\begin{enumerate}
\item For all $v_{1},v_{2}\in VG$,
\begin{equation}
H(v_{1},v_{2};0) = \lim\limits_{t \rightarrow 0}
H(v_{1},v_{2};t)=\left\{ \begin{array}{c}
1\,\,\, \textrm{ if \ensuremath{v_{1}=v_{2}}}\\
0\,\,\, \textrm{ if \ensuremath{v_{1}\neq v_{2}}}
\end{array}\right.\label{eq:dirac property of heat kernel}
\end{equation}
\item The function $L_{v_1}H(v_1, v_2;t)$ extends to a continuous function on  $VG\times VG\times [0,\infty).$
\end{enumerate}
We say the parametrix $H$ is of order $k$ if for some integer $k \geq 0,$
$$
|L_{v_1}H(v_1, v_2;t)| = O(t^k)
\,\,\,\,\,
\text{\rm as $t \to 0$.}
$$

\begin{lem}
\label{lem: L of convol} Let $H$ be a parametrix for the heat operator on $G$.
Let $f=f(v_{1},v_{2};t):VG\times VG\times\mathbb{R}_{> 0}\to\mathbb{R}$
be continuous in $t$  and bounded for all $v_{1},v_{2}\in VG$. Then
$$
L_{v_{1}}(H\ast f)(v_{1},v_{2};t)=f(v_{1},v_{2};t)+(L_{v_{1}}H\ast f)(v_{1},v_{2};t)
$$
for all $v_{1},v_{2}\in VG$ and $t\in\mathbb{R}_{>0}$.
\end{lem}

\begin{proof}
  For any $t>0$ we have
\begin{equation}\label{eq:heat_of_convolution}
L_{v_{1}}(H\ast f)=
    \frac{\partial}{\partial t}(H\ast f)
    + \Delta_{G,v_{1}}(H\ast f)=
  \frac{\partial}{\partial t}(H\ast f)+   (\Delta_{G,v_{1}}H)\ast f.
\end{equation}
By the Leibniz integration formula, the first term on the right hand side of \eqref{eq:heat_of_convolution} is equal to
\begin{align*}
\frac{\partial}{\partial t}&\int\limits _{0}^{t}\sum_{v\in VG}H(v_{1},v;t-r)f(v,v_{2};r)dr\\&
=\sum_{v\in VG}H(v_{1},v;0)f(v,v_{2};t)+\int\limits _{0}^{t}\sum_{v\in VG}
\frac{\partial}{\partial t}H(v_{1},v;t-r)f(v,v_{2};r)dr\\&
=f(v_{1},v_{2};t)+(\frac{\partial}{\partial t}H\ast f)(v_1,v_2;t).
\end{align*}
Therefore,
\begin{align*}
L_{v_{1}}(H\ast f)(v_{1},v_{2};t) &=
  \frac{\partial}{\partial t}(H\ast f)(v_{1},v_{2};t)+   (\Delta_{G,v_{1}}H\ast f)(v_{1},v_{2};t)
\\&=f(v_{1},v_{2};t)+(\frac{\partial}{\partial t}H\ast f)(v_1,v_2;t)+(\Delta_{G,v_{1}}H\ast f)(v_{1},v_{2};t)
\\&=f(v_{1},v_{2};t)+(L_{v_{1}}H\ast f)(v_{1},v_{2};t),
\end{align*}
as claimed.
\end{proof}

With all this, we now can state the main theorem in this section
\begin{thm}
\label{thm: formula for heat kernel} Let $H$ be a parametrix of
order $k\geq0$.  For $v_1,v_2 \in VG$ and $t\in \mathbb{R}_{\geq 0}$ define
\begin{equation}\label{eq:Neuman_series}
F(v_{1},v_{2};t):=\sum_{\ell=1}^{\infty}(-1)^{\ell}
(L_{v_{1}}H)^{\ast\ell}(v_{1},v_{2};t).
\end{equation}
Then the series \eqref{eq:Neuman_series}
converges absolutely and uniformly on every compact subset of $VG\times VG\times\mathbb{R}_{\geq 0}$.
Furthermore, the heat kernel $H_{G}$ on $G$
associated to graph Laplacian $\Delta_{G,v_{1}}$ is given by
\begin{equation}\label{eq:heat_kernel_parametrix}
H_{G}(v_{1},v_{2};t)=H(v_{1},v_{2};t)+(H\ast F)(v_{1},v_{2};t)
\end{equation}
with
$$
(H\ast F)(v_{1},v_{2};t)
= O(t^{k+1})
\,\,\,\,\,
\text{\rm as $t \rightarrow 0$.}
$$
\end{thm}

\begin{proof}
Set
$$
\tilde{H}(v_{1},v_{2};t) := H(v_{1},v_{2};t)+(H\ast F)(v_{1},v_{2};t).
$$
From the characterizing properties of the heat kernel$$
\lim_{t\to 0}\widetilde{H}(v_{1},v_{2};t)=\lim_{t\to 0}H(v_{1},v_{2};t)=\delta_{v_{1}=v_{2}}
,$$ it suffices to show
that
\begin{equation}\label{eq:heat_kernel_criteria}
L_{v_{1}}\tilde{H}(v_{1},v_{2};t)=0
\,\,\,\,\,
\text{\rm and}
\,\,\,\,\,
\lim_{t\to 0} \tilde{H}(v_{1},v_{2},t)=\delta_{v_{1}=v_{2}}.
\end{equation}
By Lemma \ref{lem:convergence of the series}, the series $F(v_1,v_2;t)$ defined in \eqref{eq:Neuman_series}
converges uniformly and absolutely and has order $O(t^{k})$ as $t \rightarrow 0$.  Since $H$ is bounded on any finite interval, Lemma \ref{lem:convolution bounds} then yields the asymptotic
$$
(H\ast F)(v_{1},v_{2};t)= O(t^{k+1})
\,\,\,\,\,
\text{\rm as $t \rightarrow 0$.}
$$
In particular,
$$
\lim_{t\to 0}\widetilde{H}(v_{1},v_{2};t)=\lim_{t\to 0}H(v_{1},v_{2};t)=\delta_{v_{1}=v_{2}}.
$$

It remains to prove the vanishing of $L_{v_1}\tilde{H}$ in \eqref{eq:heat_kernel_criteria}.
For this, we can apply Lemma \ref{lem: L of convol} to get that
\begin{align*}
L_{v_{1}}\tilde{H}(v_{1},v_{2};t)&=
L_{v_{1}}(H)(v_{1},v_{2};t)+L_{v_{1}}(H\ast F)(v_{1},v_{2};t)
\\&=
L_{v_{1}}(H)(v_{1},v_{2};t)+\sum_{\ell=1}^{\infty}(-1)^{\ell}(L_{v_{1}}
{H})^{\ast\ell}(v_{1},v_{2};t)\\&\,\,\,\,\,\,
+(L_{v_{1}}{H})*\left(\sum_{\ell=1}^{\infty}(-1)^{\ell}(L_{v_{1}}{H})^{\ast(\ell)}\right)(v_{1},v_{2};t)\\&=
L_{v_{1}}(H)(v_{1},v_{2};t)+\sum_{\ell=1}^{\infty}(-1)^{\ell}(L_{v_{1}}
{H})^{\ast\ell}(v_{1},v_{2};t)\\&\,\,\,\,\,\,
+\sum_{\ell=1}^{\infty}(-1)^{\ell}(L_{v_{1}}{H})^{\ast(\ell+1)}(v_{1},v_{2};t)\\&
=0,
\end{align*}
where we used absolute convergence of the series defining $F(v_{1},v_{2};t)$
to change the order of summation.  This completes the proof.
\end{proof}

\begin{rem}\rm
In \eqref{eq:heat_kernel_parametrix}, the function $H(v_{1},v_{2};t)$ is a
parametrix, so it is required to satisfy the reasonably weak conditions given in its definition.  In particular, one does not use any information about the edge structure
associated to the graph.  One incorporates the finer information regarding
the edge data through the Laplacian, which is used to form the Neumann series
\eqref{eq:Neuman_series}.  In this regard, the parametrix construction of the
heat kernel on a graph, as given in Theorem \ref{thm: formula for heat kernel}
is easier than the parametrix construction for the heat kernel in the setting
of Riemannian geometry (see, for example, Chapter IV of \cite{Ch84}, page 149) where
the Laplacian is necessary in the construction of the parametrix.
\end{rem}

\section{Graphs embedded in a domain\label{sec:Graphs-embedded-in}}

Let us assume that a graph $G$ is embedded in a Riemannian domain $\Omega$.
In this section, we will show one way in which the Riemannian heat kernel
on $\Omega$ can be used to construct the graph heat kernel on $G$.  In effect,
we define a parametrix on the graph by computing an average of the heat
kernel from $\Omega$ over neighborhoods of the vertices of $G$.  We then
employ Theorem \ref{thm: formula for heat kernel}.  Certainly, there are other
means by which one could define a parametrix.  However, we find it appealing
to develop a direct connection between the Riemannian heat kernel on $\Omega$
and the graph heat kernel on $G$.

Let $\Omega\subset\mathbb{R}^{N}$ be a bounded domain with $C^{\infty}$
boundary $\partial\Omega$.  The Laplacian on $\Omega$ is defined
through its action on twice continuously differentiable functions
$f$ on $\Omega$ as
\[
\Delta_{\Omega}f(x)=-\sum_{i=1}^{N}\frac{\partial^{2}}{\partial x_{i}^{2}}f(x).
\]
Let $K_{\Omega}(x,y;t)$ denote a heat kernel on $\Omega$,
meaning the fundamental solution to the equation
$$
(\Delta_{\Omega}+\partial_{t})K(x,y;t)=0
$$
on $\Omega\times\Omega\times(0,\infty)$ subject to the initial condition
$$
\lim_{t\to 0}K(x,y;t)=\delta_{x},
$$
 where $\delta_{x}$ is
the Dirac delta mass concentrated at $x$.  For definiteness we assume that $K_{\Omega}(x,y;t)$
satisfies a Dirichlet boundary condition, but this will play no role in our construction and any reasonable alternative boundary condition may be assumed.   
We assume that $\Omega$ is fixed and will suppress
the subscript $\Omega$ and denote a the heat kernel on $\Omega$ simply by $K(x,y;t)$.
Heat kernels with different boundary condition are, of
course, different; however, for our construction we can take any one.

Let $G$ be a graph embedded in $\Omega$, in the sense that $VG\subset\Omega$.
Let $d$ denote the Euclidean distance on $\Omega$ and let $m$ be
the Lebesgue measure also denoted $dx$ below.
For any vertex $v$ in $G$ we let
\[
N(v)=\left\{ x\in\Omega:d(x,v)<d(x,v')\textrm{ for all \ensuremath{v'\in VG}}\right\}
\]
be the (Dirichlet-) Vorono\u{\i} decomposition of $\Omega$ induced
by $VG$. In words, $N(v)$ is the set of points in $\Omega$ closer to $v$ than any other vertex in $VG$.  The collection of sets
$\{N(v):v\in VG\}$ are disjoint, by construction,
and non-empty since $v\in N(v)$. Each set is open, thus measurable.
We assume that $VG$ is embedded in $\Omega$ in such a way that all
sets $N(v)$ in the Vorono\u{\i} decomposition are path-connected.
We denote by $\mu_{v}$ the measure of $N(v)$. For example, if $\Omega$
is convex, then each set $N(v)$ is open and convex.

The set
$$
H(v,v'):=\{x\in\mathbb{R}^{N}:d(x,v)=d(x,v'): v, v'\in VG\}
$$
is a union of hyperplanes in $\mathbb{R}^{N}$, hence a measure zero set. The
set $\Omega\setminus\bigcup_{v\in VG}N(v)$ is a subset of $H(v,v')$, hence is also a
set of measure zero.

For any $\delta>0$ and $v\in VG$, define
the set
$$
N_{\delta}(v):=\{x\in N(v):d(x,\partial N(v))<\delta\},
$$
where $\partial N(v)$ stands for the boundary of $N(v)$.
For an arbitrary and fixed $\epsilon>0$, we choose $\delta>0$ which
depends on $\epsilon$ and $G$ such that
$$
N(v)\setminus N_{\delta}(v)\neq\emptyset
\,\,\,\,\,
\text{\rm and}
\,\,\,\,\,
m(N_{\delta}(v))<\epsilon/|VG|
\,\,\,
\text{\rm for all $v\in VG$.}
$$
For such a fixed $\epsilon>0$, let us introduce a family of $C^{\infty}$-functions
$\eta_{v}:\Omega\to[0,\infty)$, indexed by $v\in VG$, which satisfy the following
properties.

\begin{itemize}
\item[(a)] For all $v\in VG$ one has $\textrm{supp}(\eta_{v})\subseteq N(v)$,
meaning $\eta_{v}(x) = 0$ whenever $x \notin N(v)$.
\item[(b)] We have that $\eta_{v}(x)\equiv 1$ for all $x\in N(v)\setminus N_{\delta}(v)$  and
$\eta_{v}(x)\equiv 0$ for all $x\in N_{\delta/2}(v)$.
\item[(c)] For all $v \in VG$, we have that
$$
\int\limits _{\Omega}\eta_{v}(x)^{2}=\mu_{v}=m(N(v)).
$$
\end{itemize}
In words, the function $\eta_{v}$ is equal to one on $N(v)\setminus N_{\delta}(v)$, equal
to zero outside $N(v)\setminus N_{\delta/2}(v)$, increases from one and then decrease to
zero on $N_{\delta}(v)\setminus N_{\delta/2}(v)$ so that one has the integral condition (c).
It is an elementary exercise in real analysis to construct such a family of functions.  We note for later use the relation
\begin{equation}
  \label{eq:CS}
\int_\Omega \eta_v(x)\, dx\leq \mu_v
\end{equation}
which follows immediately from (a) and (c) above, together with the Cauchy-Schwartz inequality.

With this notation, we can use a heat kernel
$K(x,y;t)$ on $\Omega$ to construct a parametrix of any order
$k \geq0$ for the heat kernel on the graph $G$.  The details are as follows.
For any $v_{1},v_{2}\in VG$ and $t>0$, define
\begin{equation}
H_{0}(v_{1},v_{2};t)=\frac{1}{\sqrt{\mu_{v_{1}}\mu_{v_{2}}}}\int\limits _{N(v_{1})}\int\limits _{N(v_{2})}K(x,y;t)\eta_{v_{1}}(x)\eta_{v_{2}}(y)dydx.\label{eq:H0-def}
\end{equation}
From the known properties of $K$, and from the smoothness
of the family $\eta_{v}$ for $v\in VG$, the function (\ref{eq:H0-def}) is infinitely
differentiable in $t$.
Moreover, for each $v\in VG$ the function
$$
u_{v}(x;t):=\int\limits _{N(v)}K(x,y;t)\eta_{v}(y)dy=\int\limits _{\Omega}K(x,y;t)\eta_{v}(y)dy
$$
is a solution to the initial value problem
\begin{equation}\label{eq:initial value problem1}
\left(\Delta_{\Omega}+\frac{\partial}{\partial t}\right)u_{v}(x;t)=0
\,\,\,\,\,
\text{\rm for}
\,\,\,\,\,
(x,t) \in \Omega \times (0,\infty)
\end{equation}
with
\begin{equation}\label{eq:initial value problem2}
 u_{v}(x;0)=\eta_{v}(x)
 \,\,\,\,\,
\text{\rm for}
\,\,\,\,\,
x\in\Omega.
\end{equation}

\begin{prop}
\label{prop: heat bounds}
Assume the notation as above.
\begin{itemize}
\item[(i)] For all $v_{1},v_{2}\in VG$, we have that
\begin{equation}
\lim_{t\to 0}H_{0}(v_{1},v_{2};t)=\left\{ \begin{array}{c}
1\,\,\,\textrm{ if \ensuremath{v_{1}=v_{2}}}\\
0\,\,\,\textrm{ if \ensuremath{v_{1}\neq v_{2}}}
\end{array}\right.\label{eq:dirac property of heat kernel ex}
\end{equation}
\item[(ii)] For any $k\geq 0$ all $v_{1},v_{2}\in VG$, there is a constant
$C_{k}(v_{1},v_{2})$ such that for all $t>0$, we have that
\begin{equation}
\left|\frac{\partial^{k}}{\partial t^{k}}H_{0}(v_{1},v_{2};t)\right|\leq C_{k}(v_{1},v_{2})
\,\,\,\,\,
\text{\rm as $t \rightarrow 0$.}
\label{eq: Hk deriv in t bound ex}
\end{equation}
\end{itemize}
\noindent
\end{prop}

\begin{rem}\rm
From the above proposition it is evident that  $H_{0}(v_{1},v_{2};t)$, defined by \eqref{eq:H0-def} is a parametrix of order zero for the heat kernel on $G$. Moreover, one can show that for any $k\geq 0$, there is a polynomial $P_{k}(t)$ in $t$ whose coefficients depends on $v_{1}$ and $v_{2}$, such that $H_{0}(v_{1},v_{2};t) - P_{k}(t)$ is a parametrix of order $k+1$.
\end{rem}

\begin{proof}
(i) From the definition \eqref{eq:H0-def} of $H_{0}$, we have that
$$
\lim_{t\to0}H_{0}(v_{1},v_{2};t)=\lim_{t\to0}\frac{1}{\sqrt{\mu_{v_{1}}\mu_{v_{2}}}}\int\limits _{N(v_{1})}\eta_{v_{1}}(x)u_{v_{2}}(x;t)dx.
$$
Since all sets are bounded and have finite volume, the dominated convergence theorem implies that
$$
\lim_{t\to0}\int\limits _{N(v_{1})}\eta_{v_{1}}(x)u_{v_{2}}(x;t)dx=\int\limits _{N(v_{1})}\left(\eta_{v_{1}}(x)\lim_{t\to0}u_{v_{2}}(x;t)\right)dx.
$$
Hence,
$$
\lim_{t\to0}H_{0}(v_{1},v_{2};t)=\frac{1}{\sqrt{\mu_{v_{1}}\mu_{v_{2}}}}\int\limits _{N(v_{1})}\eta_{v_{1}}(x)\eta_{v_{2}}(x)dx.
$$
If $v_{1}\neq v_{2}$, the functions $\eta_{v_1}$ and $\eta_{v_2}$ have disjoint support, implying that $\lim_{t\to 0}H_{0}(v_{1},v_{2};t)=0$.
If $v_{1}=v_{2}$, we get that
$$
\lim_{t\to0}H_{0}(v_{1},v_{1};t)=\frac{1}{\mu_{v_{1}}}\int\limits _{N(v_{1})}\eta_{v_{1}}^{2}(x)dx=1,
$$
by assumption (c) regarding the function $\eta_{v}$.  This proves assertion (i).

To prove (ii),
consider two arbitrary and fixed vertices $v_{1},v_{2}\in VG$.  Since the function $\displaystyle \frac{\partial^{k}}{\partial t^{k}}H_{0}(v_{1},v_{2};t)$
is right-continuous in $t\geq0$, there exists $\delta=\delta(v_{1},v_{2},k)$
such that
\[
\left|\frac{\partial^{k}}{\partial t^{k}}H_{0}(v_{1},v_{2};t)\right|\leq1+\left|\lim_{t\to0}\left(\frac{\partial^{k}}{\partial t^{k}}H_{0}(v_{1},v_{2};t)\right)\right|,
\]
for all $t\in(0,\delta)$.
Let us fix such $\tilde{\delta}=\tilde{\delta}(v_{1},v_{2},k)$. To
prove that $\displaystyle  \left|\frac{\partial^{k}}{\partial t^{k}}H_{0}(v_{1},v_{2};t)\right|$
is bounded on $(0,\tilde{\delta})$ it suffices to show that
$\displaystyle \left|\lim_{t\to0}(\frac{\partial^{k}}{\partial t^{k}}H_{0}(v_{1},v_{2};t))\right|$
is bounded. Let us take $k=1$ first, to illustrate our approach.

From the definition \eqref{eq:H0-def} of $H_{0}$, the dominated
convergence theorem implies that
\[
\lim_{t\to0}\frac{\partial}{\partial t}H_{0}(v_{1},v_{2};t)=\sqrt{\frac{\mu_{v_{1}}}{\mu_{v_{2}}}}\frac{1}{\mu_{v_{1}}}\int\limits _{N(v_{1})}\lim_{t\to0}\left(\frac{\partial}{\partial t}u_{v_{2}}(x;t)\right)\eta_{v_{1}}(x)dx.
\]
Recall that $u_{v_{2}}(x;t)$ is a  solution to the
initial-value problem \eqref{eq:initial value problem1} and \eqref{eq:initial value problem2} with $v=v_{2}$.
As such, the continuity of $u_{v_{2}}(x;t)$ and $\Delta_{\Omega}u_{v_{2}}(x;t)$
in the $t$-variable implies that
\begin{align*}
\lim_{t\to0}\left(\frac{\partial}{\partial t}u_{v_{2}}(x;t)\right)&=\lim_{t\to0}\left(-\Delta_{\Omega}u_{v_{2}}(x;t)\right)
\\&=-\Delta_{\Omega}u_{v_{2}}(x;0)=-\Delta_{\Omega}\eta_{v_{2}}(x).
\end{align*}
Therefore,
$$
\lim_{t\to0}\frac{\partial}{\partial t}H_{0}(v_{1},v_{2};t)=\sqrt{\frac{\mu_{v_{1}}}{\mu_{v_{2}}}}\cdot\frac{1}{\mu_{v_{1}}}\int\limits _{N(v_{1})}(-\Delta_{\Omega}\eta_{v_{2}}(x))\eta_{v_{1}}(x)dx.
$$
As a result, we have that
$$
\lim_{t\to0}\frac{\partial}{\partial t}H_{0}(v_{1},v_{2};t)=0
\,\,\,\,\,
\text{if  } v_{1}\neq v_{2}
$$
and
\[
\lim_{t\to0}\frac{\partial}{\partial t}H_{0}(v_{1},v_{1};t)=\frac{1}{\mu_{v_{1}}}\int\limits _{N(v_{1})}(-\Delta_{\Omega}\eta_{v_{1}}(x))\eta_{v_{1}}(x)dx.
\]
We now can apply the mean value theorem for integrals, which is justified because
$\eta_{v}$ is non-negative and $\Delta_{\Omega}\eta_{v_{1}}(x)$
is continuous.  In doing so, we conclude that for some $x_{1}\in N(v_{1})$ we have the bounds
\[
\lim_{t\to0}\frac{\partial}{\partial t}H_{0}(v_{1},v_{1};t)\leq|\Delta_{\Omega}\eta_{v_{1}}(x_{1})|\cdot\frac{1}{\mu_{v_{1}}}\int\limits _{N(v_{1})}\eta_{v_{1}}(x)dx\leq|\Delta_{\Omega}\eta_{v_{1}}(x_{1})|,
\]
where the last inequality follows from \eqref{eq:CS}.  Therefore,
\[
\left|\lim_{t\to0}(\frac{\partial}{\partial t}H_{0}(v_{1},v_{2};t))\right|\leq\left|\Delta_{\Omega}\eta_{v_{1}}(x_{1})\right|
\cdot \delta_{v_{1}=v_{2}},
\]
for some $x_{1}\in N(v_{1})$.  For all $v \in VG$, the functions $\eta_{v}$
are smooth.  Let us define
\[
C_{k}(v):=\sup_{x\in N(v)}\left|\Delta_{\Omega}^{k}\eta_{v}(x)\right|,
\]
where $\Delta_{\Omega}^{k}$ denotes the operator $\Delta_{\Omega}$ applied
$k$ times.  Then we have that
\[
\left|\lim_{t\to0}(\frac{\partial}{\partial t}H_{0}(v_{1},v_{2};t))\right|\leq C_{1}(v_{1})\cdot \delta_{v_{1}=v_{2}}.
\]
Consequently, we have that
$\displaystyle \frac{\partial}{\partial t}H_{0}(v_{1},v_{2};t)$ is bounded by
a constant which is independent of $t$ for all $t\in(0,\tilde{\delta})$.
The proof of (ii) in the case $k=1$ is now complete.

When $k>1$, we proceed in a similar manner. Namely,
we write that
\[
\lim_{t\to0}\frac{\partial^{k}}{\partial t^{k}}H_{0}(v_{1},v_{2};t)=\sqrt{\frac{\mu_{v_{1}}}{\mu_{v_{2}}}}\frac{1}{\mu_{v_{1}}}\int\limits _{N(v_{1})}\lim_{t\to0}\left(\frac{\partial^{k}}{\partial t^{k}}u_{v_{2}}(x;t)\right)\eta_{v_{1}}(x)dx.
\]
Since the functions $\eta_{v}$ are $C^{\infty}$,
\begin{align*}
\lim_{t\to0}\left(\frac{\partial^{k}}{\partial t^{k}}u_{v_{2}}(x;t)\right)&=\lim_{t\to0}\left((-1)^{k}\Delta_{\Omega}^{k}u_{v_{2}}(x;t)\right)
\\&=(-1)^{k}\Delta_{\Omega}^{k}u_{v_{2}}(x;0)=(-1)^{k}\Delta_{\Omega}^{k}\eta_{v_{2}}(x),
\end{align*}
and arguing as above, we get
\[
\left|\lim_{t\to0}(\frac{\partial^{k}}{\partial t^{k}}H_{0}(v_{1},v_{2};t))\right|\leq C_{k}(v_{1})\cdot \delta_{v_{1}=v_{2}}.
\]
Therefore  $\displaystyle \frac{\partial^{k}}{\partial t^{k}}H_{0}(v_{1},v_{2};t)$
is bounded
by a constant which is independent of $t$ for all $t\in(0,\tilde{\delta})$.
For any fixed $\tilde{\delta}>0$, the function $\displaystyle \frac{\partial^{k}}{\partial t^{k}}K(x,y;t)$
is uniformly bounded on the set $\Omega\times\Omega\times[\tilde{\delta},\infty]$,
by a constant $C_{\tilde{\delta},k}$, see \cite{Gr09}. Therefore, for all $v_{1},v_{2}\in VG$ and $t\geq\tilde{\delta}$,
\begin{equation}
  \left|\frac{\partial^{k}}{\partial t^{k}}H_{0}(v_{1},v_{2};t)\right|\leq C_{\tilde{\delta},k}\frac{1}{\sqrt{\mu_{v_{1}}\mu_{v_{2}}}}\int\limits _{N(v_{1})}\int\limits _{N(v_{2})}\eta_{v_{1}}(x)\eta_{v_{2}}(y)\,dy\,dx
  \leq C_{\tilde{\delta},k}m(\Omega)\label{eq: bound for deriv away from zero}
\end{equation}
by \eqref{eq:CS}.
We have shown that $\displaystyle \frac{\partial^{k}}{\partial t^{k}}H_{0}(v_{1},v_{2};t)$ is bounded
by a constant independent of $t$ for all $t\in[\tilde{\delta},\infty)$,
and the proof of the proposition is complete.
\end{proof}

\begin{rem}
The normalization factor $\frac{1}{\sqrt{\mu_{v_{1}}\mu_{v_{2}}}}$
in the definition \eqref{eq:H0-def} was chosen so that
\[
H_{0}(v_{1},v_{2};t)=H_{0}(v_{2},v_{1};t).
\]
If one defines
\begin{equation}
\mathcal{H}_{0}(v_{1},v_{2};t):=\frac{1}{\mu_{v_{1}}}\int\limits _{N(v_{1})}\int\limits _{N(v_{2})}K(x,y;t)\eta_{v_{1}}(x)\eta_{v_{2}}(y)dydx.,\label{eq:H0-def-updated}
\end{equation}
it is easy to see that $\mathcal{H}_{0}(v_{1},v_{2};t)$ is also a
parametrix of order zero for the heat kernel on $G$.
The  different normalization in \eqref{eq:H0-def-updated} can be justified
by the fact that in the construction of the heat kernel on $G$ the
graph Laplacian acts only to the first variable and the formula \eqref{eq:H0-def-updated}
can be written as
\[
\mathcal{H}_{0}(v_{1},v_{2};t)=\frac{1}{\mu_{v_{1}}}\int\limits _{N(v_{1})}\eta_{v_{1}}(x)u_{v_{2}}(x;t)dx,
\]
which is exactly the ``averaged by $\eta_{v_{1}}$''
discretization (in the sense of Definition 4.1. on p. 689 of \cite{BIK14})
of the solution $u_{v_{2}}(x;t)$ to the initial value problem \eqref{eq:initial value problem1}
and \eqref{eq:initial value problem2}.
\end{rem}

\begin{rem}\label{rem:Domain_to_graph}
A particularly interesting special case of Theorem \ref{thm: formula for heat kernel}
is when $k=0$.  Specifically, we have that the heat kernel
on the graph can be written as
\begin{equation}
H_{G}(v_{1},v_{2};t)=H_{0}(v_{1},v_{2};t)+(H_{0}\ast F)(v_{1},v_{2};t),\label{eq:heat kernel with k=00003D0}
\end{equation}
where $H_{0}(v_{1},v_{2};t)$ is defined by \eqref{eq:H0-def}.
This formula is particularly nice because $\lim_{t\to0}(H_{0}\ast F)(v_{1},v_{2};t)=0$
for all $v_{1},v_{2}\in VG$.
\end{rem}


\section{Heat kernels for subgraphs\label{sec:Heat-kernels-subgraphs}}

Let $G$ be a finite or infinite subgraph of a finite or infinite
weighted graph $\tilde{G}$ with the weight function $\tilde{w}_{xy}$
for vertices $x,y\in V\tilde{G}$.  For us, a subgraph $G$ of $\tilde{G}$ is
obtained by using  a subset $VG$ of the vertices $V\tilde{G}$
defining the weight function $w_{xy}$ for $x,y \in VG$ to be equal either to the same weight $\tilde{w}_{xy}$ on $\tilde{G}$ or set to zero.
In the case when $\tilde{w}_{xy}\neq 0$ but $w_{xy}=0$, we say that we have removed the edge connecting $x$ and $y$ in $\tilde{G}$.  We define
\begin{equation}
  \label{defn:valence}
\tilde\mu(x)=\sum_{y\in V\tilde{G}}\tilde{w}_{xy}\mbox{\ and\ } \mu(x)=\sum_{y\in VG}w_{xy}.
\end{equation}
When the graph $\tilde{G}$ is infinite, we assume that
for all $x\in V\tilde{G}$, the function $\tilde{w}(x,\cdot)$ has finite support
and that the quantity $\tilde\mu(x)$ is uniformly
bounded for $x \in V\tilde{G}$.
In this case, we say that $\tilde{G}$ has \emph{bounded valence}.  We further assume that $\tilde G$ is of bounded degree.
Let $\partial G$ denote the set of \emph{boundary vertices} of the subgraph
$G$ in $\tilde{G}$.  That is, $\partial G$ consists of vertices $x\in VG$ such that
$\mu(x)<\tilde\mu(x).$
We call $\mathrm{Int}(G)=VG\setminus\partial G$ the
 set of interior points of $G$.

Let $H_{\tilde{G}}$ denote the heat kernel for the graph $\tilde{G}$
with respect to the weighted graph Laplacian $\Delta_{\tilde{G}}$.  Recall that
for any function $f:V\tilde{G}\to\mathbb{C}$
\[
\Delta_{\tilde{G}}f(x)=\sum_{y\in\tilde{G}}(f(x)-f(y))\tilde{w}_{xy}.
\]
The Laplacian $\Delta_{\tilde{G}}$
is a non-negative and self-adjoint operator on the Hilbert space $L^{2}(V\tilde{G})$.
Under our assumptios of finite valence and finite degree, the existence and uniqueness
of the heat kernel follows from \cite{DM06}.

Let $L_{\tilde{G},v_{1}}=\Delta_{\tilde{G},v_{1}}+\partial_{t}$ be
the heat operator on $\tilde{G}$ acting in the first variable, and
let $L_{G,v_{1}}=\Delta_{G,v_{1}}+\partial_{t}$ be the heat operator
on $G$.
Define
$$
H(v_{1},v_{2};t):=H_{\tilde{G}}(v_{1},v_{2};t)
\,\,\,\,\,
\text{\rm for $v_{1},v_{2}\in VG$ and $t\geq0$.}
$$
In other words, $H$ is the restriction
of the heat kernel from $\tilde{G}$ to $G$.  Clearly, if  $v_1\in \mathrm{Int}(G)$ we have
$L_{G,v_{1}}H(v_{1},v_{2};t)=0$, but if $v_1\in\partial G$ then
\begin{equation} \label{eq: L on boundary points}
L_{G,v_{1}}H(v_{1},v_{2};t)=
-\sum_{v\in A_{\tilde{G},G}(v_{1})}\left(H(v_{1},v_{2};t)-H(v,v_{2};t)\right)w_{v_{1}v},
\end{equation}
where
$$
A_{\tilde{G},G}(v_{1})=\{v\in V\tilde{G}\setminus VG: \, v\, \text{is adjacent to }v_1 \text{ in } \tilde{G}\} \cup \{v\in VG :\, \tilde{w}_{vv_1}\neq 0\, \text{ but }\, w_{vv_1}=0\}.
$$
The set $A_{\tilde{G},G}(v_{1})$
is not empty for $v_{1}\in\partial G$, so one cannot expect that the difference
\eqref{eq: L on boundary points}
to vanish so $H$ is certainly not the heat kernel on $G$.
However, we show in the following proposition that $H(v_{1},v_{2};t)$ is a
parametrix of order zero for the heat kernel on $G$.
\begin{prop}
\label{prop: heat kernel on the cover parametrix}
Assume the notation as above.
\begin{itemize}
\item[(i)] For all $v_{1},v_{2}\in VG$, we have that
\begin{equation}
\lim_{t\to 0}H(v_{1},v_{2};t)=\left\{ \begin{array}{c}
1\,\,\, \textrm{ if \ensuremath{v_{1}=v_{2}}}\\
0\,\,\, \textrm{ if \ensuremath{v_{1}\neq v_{2}}}.
\end{array}\right.
\label{eq:dirac property of heat kernel ex2}
\end{equation}
\item[(ii)] For $v_{1},v_{2}\in VG$, there is a constant
$C(v_{1},v_{2})$ such that for all $t>0$, we have that
\begin{equation}
\left|\frac{\partial}{\partial t}H(v_{1},v_{2};t)\right|\leq C(v_{1},v_{2}).\label{eq: Hk deriv in t bound ex2}
\end{equation}
\end{itemize}
\end{prop}

\begin{proof}
Part (i) follows trivially from the fact that $H(v_{1},v_{2};t)$
is the restriction of the heat kernel on $\tilde{G}$, so it remains to prove part (ii).

Since $H(v_{1},v_{2};t)=H_{\tilde G}(v_{1},v_{2};t) $ for $v_1,v_2\in VG,$
\[
  \frac{\partial}{\partial t}H(v_{1},v_{2};t)=-\Delta_{\tilde{G},v_{1}}H(v_{1},v_{2};t)
  =\sum_{y\in VG}H(y,v_2;t)\tilde w_{v_1y}-H(v_1, v_2;t)\mu(v_1).
\]
Both $\tilde w$ and $\mu$ are bounded, so
in order to prove \eqref{eq: Hk deriv in t bound ex2}, it suffices to show that $H(v_{1},v_{2};t)$
is bounded as a function of $t$ for fixed $v_{1},v_{2}$ by some
constant, which may depend upon $v_{1}$ and $v_{2}$.  We could not find a convenient reference for this fact so we provide a short
proof for the sake of completeness.

The Dirac property \eqref{eq:dirac property of heat kernel ex2} implies
that $H(v_{1},v_{2};t)$ is bounded by an explicit constant, such as $2$,
when viewed as a function of
$t$ for $t\in[0,\eta)$ for some $\eta>0$.  It remains to prove boundedness
for $t\in[\eta,\infty)$. In the case when $V\tilde{G}$ is finite, we can use
the spectral expansion of the heat
kernel as a sum over eigenfunctions $\psi_{j}$ of the Laplacian $\Delta_{\tilde{G}}$
and the Cauchy-Schwartz inequality
to deduce that for $t\geq\eta$ one has
\begin{align}\label{eq:offdiagonal_bound}\nonumber
|H(v_{1},v_{2};t)|&\leq\left(\sum_{j}e^{-\lambda_{j}t}|\psi_{j}(v_{1})|^{2}\right)^{1/2}\left(\sum_{j}e^{-\lambda_{j}t}|\psi_{j}(v_{2})|^{2}\right)^{1/2}
\\& \leq\left(H(v_{1},v_{1};\eta)H(v_{2},v_{2};\eta)\right)^{1/2}.
\end{align}
Note that the fact that the eigenvalues $\lambda_j$ are nonnegative implies that each $e^{-\lambda_j t}$ above is nonincreasing on $t>0$.  If $V\tilde{G}$ is infinite, one re-writes the sums in \eqref{eq:offdiagonal_bound}
as integrals over the appropriate spectral measures, and one gets that
$H(v_{1},v_{2};t)$ is bounded for $t\geq\eta$
by $\left(H(v_{1},v_{1};\eta)H(v_{2},v_{2};\eta)\right)^{1/2}$.  This completes
the proof of the claim.
\end{proof}

If the graph $G$ is a \emph{finite} subgraph of $\tilde{G}$, Theorem
\ref{thm: formula for heat kernel} applies to deduce the heat kernel
on $G$ from the parametrix $H(v_{1},v_{2};t)$. However, the parametrix construction of
the heat kernel on infinite subgraph of a finite valence infinite
graph $\tilde{G}$ can also be carried out, under the assumption that
the boundary $\partial G$ of $G$ is finite.  The precise
result is the following proposition.

\begin{prop} \label{prop: heat kernel subgraph of inf graph}
Let $G$ be an infinite subgraph of a finite valence, bounded degree graph $\tilde{G}$
such that boundary $\partial G$ is finite. Let $\tilde{H}(v_{1},v_{2};t)$
denote the heat kernel on $\tilde{G}$. Then the series
\begin{equation}
\sum_{\ell=1}^{\infty}(-1)^{\ell}(L_{G,v_{1}}\tilde{H})^{\ast\ell}(v_{1},v_{2};t)\label{eq:defining series for Ftilde}
\end{equation}
converges absolutely and uniformly on any compact subset of $VG\times VG\times[0,\infty)$
and defines a function $\tilde{F}(v_{1},v_{2};t)$ on $VG\times VG\times[0,\infty)$.
Furthermore, the heat kernel $H_{G}(v_{1},v_{2};t)$ on $G$ is given by
$$
H_{G}(v_{1},v_{2};t):=\tilde{H}(v_{1},v_{2};t)+(\tilde{H}\ast\tilde{F})(v_{1},v_{2};t).
$$
\end{prop}

\begin{proof}
We start by observing that the function $L_{G,v_{1}}\tilde{H}(v_{1},v_{2};t)$
when viewed as a function of the first variable is supported on the
finite set $\partial G$. Moreover, the heat kernel $\tilde{H}(v_{1},v_{2};t)$ is bounded. Hence,
there exists a constant $C$ depending on the finite subset $\partial G$
of $G$ and the valence of $G$ so that
\begin{align*}
(L_{G,v_{1}}\tilde{H}\ast L_{G,v_{1}}\tilde{H})(v_{1},v_{2};t) & =\sum_{v\in\partial G}\int\limits _{0}^{t}L_{G,v_{1}}\tilde{H}(v_{1},v;t-\tau)L_{G,v_{1}}\tilde{H}(v,v_{2};\tau)d\tau\\
 & \leq |\partial G|\cdot C t,
\end{align*}
for all $(v_{1},v_{2})\in\partial G\times VG$ and all $t\geq0$, where
$|\partial G|$ denotes the number of elements in $\partial G$.
By arguing analogously as in the proof of Lemma \ref{lem:convergence of the series},
we conclude that the series \eqref{eq:defining series for Ftilde} converges
absolutely and uniformly on any compact subset of $VG\times VG\times[0,\infty)$.
The Dirac property $\lim_{t\to0}\tilde{H}(v_{1},v_{2};t)=\delta_{v_{1}=v_{2}}$
when combined with the finiteness of support of $L_{G,v_{1}}\tilde{H}(v_{1},v_{2};t)$
is sufficient to conclude that Lemma \ref{lem: L of convol} also holds
 with $H=\tilde{H}$.  The method of the proof of
Theorem \ref{thm: formula for heat kernel} then applies to show
that $H_{G}(v_{1},v_{2};t)$ is the heat kernel on $G$.
\end{proof}

\section{A subgraph of the complete graph on $N$ vertices}\label{sec: heat kernels on subgraphs}

In this section we will develop explicitly our results in the setting of a subgraph of the
complete graph on $N$ vertices. Throughout this section we will employ some of the same
considerations as in \cite{LNY21}, namely in the explicit computation of the graph convolution.
In doing so, we will show the relationship between our Theorem \ref{thm: formula for heat kernel} and
Theorem 3.3 of \cite{LNY21}.

Let $K_{N}$ denote the complete graph $K_{N}$ on $N$ vertices. Let us label the vertices
with the integers $1, 2, \cdots, N$ and define the weights to be
$\tilde{w}_{xy}=1$ for all $x\neq y$.  In the notation of the previous section, $K_N$ plays the role of
$\tilde{G}$.
The heat kernel on $K_{N}$ is well known.
Indeed, if we set $H_{\tilde G}:=H_{K_{N}}$, then we have that
\begin{equation}\label{eq: heat K on compl graph}
H_{\tilde G}(x,y;t):=\left\{ \begin{array}{ll}
\frac{1}{N}+\left(1-\frac{1}{N}\right)e^{-Nt}, & \text{ for } x=y\\
\frac{1}{N}-\frac{1}{N}e^{-Nt}, & \text{ otherwise.}
\end{array}\right.
\end{equation}
Let $G$ be any subgraph of $K_N$.  Without loss of generality, we assume that $G$ has $N$ vertices, and take $H_{\tilde G}$ as a parametrix for the heat kernel $H_G$ on $G$.
In order to compare our results with those from \cite{LNY21}, we will borrow and adapt some notation from \cite{LNY21}.

For any vertex $x\in VG$, let $d_x^G$ be the degree of $x$ in $G$ and set $d_x^{c}=(N-1)- d_v^G$.
For two vertices $x\in VG$ and $y \in VK_N$, let us write $x \sim_c y$ if $x\neq y$ and $x$ and $y$ are not connected in $G$.  Define the combinatorial Laplacian on the complement of $G$ in $K_N$ by
\begin{equation}
  \label{def:Laplacian on complement}
\Delta^cf(x)=\sum_{y\sim_c x} (f(x)-f(y))
\end{equation}
From \eqref{eq: L on boundary points} we have that
$$
L_1 H_{\tilde G}(x,y;t)=- \sum_{z\sim_c x} \left(H_{\tilde G}(x,y;t) -H_{\tilde G}(z,y;t)\right),
$$
for all $x,y\in VG$ and any $t>0$.  Note that we have shortened the notation by lettng $L_1$ denote the heat operator of $G$ acting on the first vertex variable $x$.
From the expression  \eqref{eq: heat K on compl graph} for the heat kernel $H_{\tilde G}$, we immediately have
for $x,y \in VG$ that
\begin{equation}
L_1 H_{\tilde G}(x,y;t)=\left\{
                 \begin{array}{ll}
                   -e^{-Nt} d_{x}^{c}, & \text{  if  } x=y: \\
                   e^{-Nt}, & \text{  if  } x\sim_c y; \\
                   0, & \text{  otherwise.}
                 \end{array}
               \right.\label{eq:L1_computation}
\end{equation}
In the notation of formula (3.4) from \cite{LNY21},
\begin{equation}
  \label{eq:u}
 L_1H_{\tilde G}(x,y;t) = -e^{-Nt}u_y^{c}(x).
\end{equation}
(We write $u_y^{c}$ for $u_y^{G^c}$ in \cite{LNY21}.)

In Section 3 of \cite{LNY21} the authors define an operator $T$ which acts on functions $F:VG\times VG\times\R_{>0}$ by
\begin{align}
  \label{def:T}
  TF(x,y;t)&=\int_0^t e^{-N(t-s)}
\Delta_x^c F (x,y;s)\, ds\\ \nonumber
  &=\int_0^t e^{-N(t-s)}(d_x^cF(x,y;s)-\sum_{\substack{v\in VG\\v\sim_c x}} F(v,y;s)
 ) \,ds.
\end{align}
In other words
$$TF(x,y;t)=-(L_1H_{\tilde G}\ast F)(x,y;t).
$$
This explains the connection between our Theorem \ref{thm: formula for heat kernel}:
$$H_G(x,y;t) = H_{\tilde G}(x,y;t)+\sum_{l=1}^\infty (-1)^l ((L_1H_{\tilde G})^{\ast l}\ast H_{\tilde G})(x,y;t)
$$
and equation (3.3) of \cite{LNY21}:
$$H_G(x,y;t) = \sum_{l=0}^\infty T^lH_{\tilde G} (x,y;t).
$$
Theorem 3.3 of \cite{LNY21} goes further to show that
 $$
 H_G(x,y;t)= H_{\tilde G}(x,y;t)+ te^{-Nt} u_y^{c}(x)+ \sum_{\ell=2}^{\infty} (-1)^{\ell -1} \frac{t^\ell}{\ell!}e^{-Nt} \left( \Delta_x^{c} \right)^{\ell-1}  u_y^{c}(x),
 $$

We can give a slightly different expression for $H_G(x,y;t)$ using matrix notation.
Since $G$ is undirected we have
$u_x^{c}(y)=u_y^{c}(x)$ for all $x,y\in VG$.
Let $B$ denote the $N\times N$ symmetric matrix with elements $b_{xy}=u_x^{c}(y)=u_y^{c}(x)$. Then starting with \eqref{eq:u}, it is easy to inductively show that
$$\left(H_{\tilde G}\ast (L_1H_{\tilde G})^{\ast \ell}\right)(x,y;t)=(-1)^{\ell}\frac{t^{\ell}}{\ell!}e^{-Nt}B^{\ell}(x,y),
$$
Therefore,
$$
H_G(x,y;t)= H_{\tilde G}(x,y;t)+ e^{-Nt}\sum_{\ell=1}^{\infty}\frac{t^\ell}{\ell!}B^{\ell}(x,y).
$$

Let us highlight the example when $G$ is a subgraph of $K_N$ obtained by deleting one edge from $K_{N}$.

\begin{example}\rm
Let $G_{1,2}$ be the graph that results from deleting the edge between
$1$ and $2$ in $K_N$. In this case, $B=(b_{ij})_{N\times N}$ with $b_{11}=b_{22}=1$; $b_{12}=b_{21}=-1$ and all other entries of $B$ are zeros.
Inductively, it is straightforward to show for any $\ell\geq 1$ we have that
\[
B^{\ell}(x,y)=\left\{
                       \begin{array}{ll}
                         2^{\ell-1}, & \text{ if  } x=y\in\{1,2\}, \\
                         -2^{\ell-1}, & \text{ if  } x,y\in\{1,2\},\, x\neq y, \\
                         0, & \text{ otherwise.}
                       \end{array}
                     \right.
\]
Therefore
\[
H_{G}(x,y;t)=H_{\tilde G}(x,y;t)
\]
except when $x,y\in\left\{ 1,2\right\}$, in which case we have that
\begin{align*}
H_{G}(x,x;t)&=H_{\tilde G} (x,x;t)+e^{-Nt}\sum_{\ell=1}^{\infty}2^{\ell-1}\frac{t^{\ell}}{\ell!}\\&=H_{\tilde G} (x,x;t)+e^{-Nt}(e^{2t}-1)/2
\end{align*}
and
\begin{align*}
H_{G}(x,y;t)&=H_{\tilde G}  (x,y;t)-e^{-Nt}\sum_{\ell=1}^{\infty}2^{\ell-1}\frac{t^{\ell}}{\ell!}\\&=H_{\tilde G} (x,y;t)-e^{-Nt}(e^{2t}-1)/2
\,\,\,\,\,
\text{\rm for $x\neq y$.}
\end{align*}

These computations show very explicitly and simply the effect on the heat kernel
from deleting an edge in the complete graph.  Specifically, one sees the change by a
negative number in $K_{G}(1,2;t)$ reflecting the disconnect between
the vertices $1$ and $2$.
 \end{example}

\section{Heat kernel on a discrete half-line}\label{sec: half-line}

The following example illustrates the construction of the heat kernel
on an infinite subgraph with finite boundary.

\begin{example}
\label{ex: discrete half line} Let $G$ be a discrete half-line beginning
with $0$.  More precisely, $G$ is the graph with vertices consisting of non-negative integers
such that each positive integer $j$ is connected only to its neighbors
$j+1$ and $j-1$, while $0$ is connected only to $1$. Then $G$
is an infinite subgraph of $\tilde{G}=\mathbb{Z}$ with the finite boundary $\partial{G}=\{0\}$.
Therefore Proposition \ref{prop: heat kernel subgraph of inf graph} applies to give an expression for the heat
kernel on $G$ by using the heat kernel $H(v,w;t)$ on $\mathbb{Z}$ as a parametrix.
 Let us
carry out these calculations in detail.

Recall that $H(v,w;t)=e^{-2t}I_{v-w}(2t)$;  see \cite{KN06}. The heat operator of $G$ acts on the $\Z$ heat kernel $H$ as
$$
L_{G,v}H(v,w;t)=\left\{ \begin{array}{ll}
e^{-2t}(I_{w+1}(2t)-I_{w}(2t)), & \text{ if }v=0;\\
0, & \text{ otherwise.}
\end{array}\right.
$$
Hence, the heat kernel on the discrete half-line $G$ is given by
\begin{equation}\label{eq: heat on half-line}
H_{G}(v,w;t)=e^{-2t}I_{|v-w|}(2t)+\sum_{\ell=1}^{\infty}(-1)^{\ell}\left(H\ast(L_{G,v}H)^{\ast\ell}\right)(v,w;t).
\end{equation}
We will now compute explicitly the convolution series on the right-hand side of \eqref{eq: heat on half-line}. We start by observing that
\begin{align}
  \left(H\ast(L_{G,v}H)\right)(v,w;t)&= \int\limits_{0}^t H(v,0;t-\tau)L_{G,v}H(0,w;\tau)d\tau \notag\\
  &= \frac{1}{2}e^{-2t}\int\limits_{0}^{2t}I_v(2t-u)\left(I_{w+1}(u) - I_w(u)\right)du.\label{eq: h conv lh}
\end{align}

Next, we derive a formula for the classical convolution of $I$-Bessel functions.
Let $m,n$ be non-negative integers. For $x\geq0$ we start with formula (1) on page 379 of \cite{Wa66}.
In the notation from \cite{Wa66} we take  $\mu=m$ and $\nu=n$ and
complex $z$ to get that
\[
z\int\limits _{0}^{\pi/2}J_{m}(z\cos^{2}\phi)J_{n}(z\sin^{2}\phi)\sin\phi\cos\phi d\phi=\sum_{k=0}^{\infty}(-1)^{k}J_{m+n+2k+1}(z),
\]
where $J$ stands for the classical $J$-Bessel function of the first kind.
By replacing $z$ with $iz$ and applying the identity $J_{n}(iz)=i^{n}I_{n}(z)$
(see \cite{GR07}, formula 8.406.3), the above equation becomes
\[
z\int\limits _{0}^{\pi/2}I_{m}(z\cos^{2}\phi)I_{n}(z\sin^{2}\phi)\sin\phi\cos\phi d\phi=\sum_{k=0}^{\infty}I_{m+n+2k+1}(z).
\]
The change of variables $t=z\sin^{2}\phi$ gives that
\begin{equation}
  \int\limits _{0}^{x}I_{n}(t)I_{m}(x-t)dt=2\sum_{k=0}^{\infty}I_{m+n+2k+1}(x).\label{eq: I-Bessel convolution formula}
\end{equation}
Combining \eqref{eq: h conv lh} and \eqref{eq: I-Bessel convolution formula} we arrive at
$$
\left(H\ast(L_{G,v}H)\right)(v,w;t)=e^{-2t} \sum_{k=1}^{\infty} (-1)^k I_{v+w+k}(2t)
$$
Applying the identity \eqref{eq: I-Bessel convolution formula} once again, it is immediate to see that
\begin{align*}
\left(H\ast(L_{G,v}H)^{\ast 2}\right)(v,w;t)&=\left(\left(H\ast(L_{G,v}H)\right)\ast(L_{G,v}H)\right) (v,w;t)\\&=e^{-2t} \sum_{(k_1,k_2)\in\mathbb{N}^2} (-1)^{k_1+k_2} I_{v+w+k_1+k_2}(2t).
\end{align*}
A simple inductive argument yields that, for any $\ell\geq 1$ the following identity holds true
$$
\left(H\ast(L_{G,v}H)^{\ast \ell}\right)(v,w;t)=e^{-2t} \sum_{(k_1,\ldots k_{\ell})\in\mathbb{N}^{\ell}} (-1)^{k_1+\ldots k_{\ell}} I_{v+w+k_1+\ldots k_{\ell}}(2t).
$$
Therefore,
$$
\sum_{\ell=1}^{\infty}(-1)^{\ell}\left(H\ast(L_{G,v}H)^{\ast\ell}\right)(v,w;t)= e^{-2t} \sum_{n=1}^{\infty} a(n) I_{v+w+n}(2t),
$$
where $a(1)=1$ and
$$
a(n)=(-1)^n \sum\limits_{\ell=1}^{n} (-1)^{\ell} |S(n,\ell)|, \quad \text{for} \quad n\geq 2,
$$
where $|S(n,\ell)|$ is the cardinality of the set $S(n,\ell)=\left\{(k_1,\ldots k_{\ell})\in\mathbb{N}^{\ell} :\, k_1+\ldots k_{\ell}=n \right\}$. Actually, $|S(n,\ell)|$ is the number of ordered partitions of $n\geq 2$ into $\ell\geq 1$ parts and it equals $\binom{n-1}{\ell-1}$, hence
$$
a(n)=(-1)^n \sum\limits_{\ell=1}^{n} (-1)^{\ell}\binom{n-1}{\ell-1}=0,\quad \text{for} \quad n\geq 2.
$$
This proves that
$$
\sum_{\ell=1}^{\infty}(-1)^{\ell}\left(H\ast(L_{G,v}H)^{\ast\ell}\right)(v,w;t)= e^{-2t} I_{v+w+1}(2t),
$$
hence the heat kernel on a discrete half-line is given by
$$
H_{G}(v,w;t)=e^{-2t}\left(I_{|v-w|}(2t)+I_{v+w+1}(2t) \right).
$$
\end{example}


\section{The Dirichlet heat kernel on a graph}\label{sec:Dirichlet-heat-kernel}

Let $G$ be a subgraph of a (possibly infinite) graph $\tilde{G}$ of finite and bounded
valence. As above, we denote the set of boundary
points of $G$ by $\partial G$.
We denote the set of interior points of $G$ by $\mathrm{Int}(G)=VG\setminus\partial G$.
Similarly, we may view
the graph with the vertex set $VG\setminus\partial G$ and edges between
the elements of the vertex set inherited from $G$ to be a subgraph
of $G$. The boundary of this subgraph will be denoted by $\partial(G\setminus\partial G)$.

We define
the Dirichlet heat kernel $H_{D,G}(v_{1},v_{2};t):VG\times VG\times[0,\infty)$
on $G$ with boundary $\partial G$
in analogy to the Dirichlet problem on a manifold with the boundary,
or on a regular domain; see, for example Chapter VII of  \cite{Ch84}.
Specifically, we define $H_{D,G}(v_{1},v_{2};t)$
to be the solution to the
differential equation
\[
L_{G,v_{1}}H_{D,G}(v_{1},v_{2};t):=\left(\Delta_{G,v_{1}}+\frac{\partial}{\partial t}\right)H_{D,G}(v_{1},v_{2};t)=0,\text{ if }v_{1}\in G\setminus\partial G
\]
with the property that
$$
\lim_{t\to0}H_{D,G}(v_{1},v_{2},t)=\delta_{v_{1}=v_{2}}
\,\,\,\,\,
\text{\rm for $v_{1}\in G\setminus\partial G$}
$$
together with the Dirichlet boundary condition
$$
H_{D,G}(v_{1},v_{2};t)=0
\,\,\,\,\,
\text{\rm for all $v_{2}\in G$ and $t\geq 0$ whenever $v_{1}\in\partial G$.}
$$
The existence and uniqueness of the Dirichlet heat kernel readily follows from results in \cite{Hu12}.
The Dirichlet heat kernel in the setting of a discrete interval was constructed from the
heat kernel on $\mathbb{Z}$ in \cite{Do12} using a symmetry based method of images.

\begin{prop}
\label{prop: dirichlet heat kernel} Let $\tilde{G}$ be a finite
or infinite graph with bounded valence. Let $G$ be a subgraph of
$\tilde{G}$ such that boundaries $\partial G$ and $\partial (G\setminus \partial G)$ are finite. Let $\tilde{H}(v_{1},v_{2};t)$
denote the heat kernel on $\tilde{G}$.  For $(v_{1},v_{2};t)\in VG\times VG\times[0,\infty)$
define the function
\[
H_{0}(v_{1},v_{2};t):=\left\{ \begin{array}{ll}
\tilde{H}(v_{1},v_{2};t), & \text{ for }v_{1}\in VG\setminus\partial G,v_{2}\in VG,t\geq0\\
0, & \text{ otherwise.}
\end{array}\right.
\]
Then the series
\begin{equation}
\sum_{\ell=1}^{\infty}(-1)^{\ell}(L_{G,v_{1}}H_{0})^{\ast\ell}(v_{1},v_{2};t)\label{eq:defining series for F}
\end{equation}
converges absolutely and uniformly on any compact subset of $VG\times VG\times[0,\infty)$.
Furthermore, the Dirichlet heat kernel $H_{D,G}(v_{1},v_{2};t)$ on $G$ is given by
\begin{equation}\label{eq:Dirichlet_heat_kernel}
H_{D,G}(v_{1},v_{2};t):=H_{0}(v_{1},v_{2};t)+(H_{0}\ast F)(v_{1},v_{2};t).
\end{equation}
\end{prop}

\begin{proof}
First, we note that our assumptions on $\tilde{G}$ ensure existence
and uniqueness of $\tilde{H}$; see for example \cite{DM06}).  Hence,  the function
$H_{0}$ is well defined.  Let $L_{\tilde{G},v_{1}}$ denote
the heat operator on $\tilde{G}$.  Then it is immediate that for all
vertices $v_{1}\in VG\setminus(\partial G\cup\partial(G\setminus\partial G))$,
all vertices $v_{2}\in VG$ and any $t\geq 0$ one has that
\[
L_{\tilde{G},v_{1}}H_{0}(v_{1},v_{2};t)=L_{G,v_{1}}\tilde{H}(v_{1},v_{2};t)=0.
\]
In particular, this
shows that $L_{G,v_{1}}H_{0}(v_{1},v_{2};t)$ is supported on
$(\partial G\cup\partial(G\setminus\partial G))\times VG\times[0,\infty)$.
Therefore, for all $\ell\geq1$, the convolution $(L_{G,v_{1}}H_{0})^{\ast\ell}(v_{1},v_{2};t)$
is supported on the set $(\partial G\cup\partial(G\setminus\partial G))\times VG\times[0,\infty)$.

Choose an arbitrary, but fixed finite subset $A$ of $VG$.
We now will
prove for any $T>0$ the convergence of the series \eqref{eq:defining series for F}
on the set $(\partial G\cup\partial(G\setminus\partial G))\times A\times[0,\infty)$,
and that the convergence is absolute and uniform on the set.

From Proposition \ref{prop: heat kernel on the cover parametrix}, the functions $H_{0}(v_{1},v_{2};t)$
and $\displaystyle \frac{\partial}{\partial t}H_{0}(v_{1},v_{2};t)$
are uniformly bounded in $t$.  Denote an upper bound the two functions by $C(v_{1},v_{2})$.
Therefore, for every finite subset $A$ of $VG$ the function
$L_{G,v_{1}}H_{0}(v_{1},v_{2};t)$ is bounded by a certain constant
$C(A)$ which depends only on the graph $G$ and the set $A$, and
can be chosen to be a constant depending on the graph times $\max_{v_{1}\in(\partial G\cup\partial(G\setminus\partial G)),v_{2}\in A}\{C(v_{1},v_{2})\}$.

Let $C$ denote the cardinality of the set $\partial G\cup\partial(G\setminus\partial G)$.
Then
\begin{align*}
(L_{G,v_{1}}H_{0}\ast L_{G,v_{1}}H_{0})(v_{1},v_{2};t) & =\sum_{v\in\partial G\cup\partial(G\setminus\partial G)}\int\limits _{0}^{t}L_{G,v_{1}}H_{0}(v_{1},v;t-\tau)L_{G,v_{1}}H_{0}(v,v_{2};\tau)d\tau\\
 &= O\left( C\cdot C(A)t\right),
\end{align*}
where the implied constant depends only upon $G$.  By proceeding
as in the proof of Lemma \ref{lem:convergence of the series}, we deduce
that
\[
(L_{G,v_{1}}H_{0})^{\ast\ell}(v_{1},v_{2};t)=O\left( C\cdot\frac{(C\cdot C(A)t)^{\ell-1}}{(\ell-1)!}\right),
\]
where as before where the implied constant depends only upon $G$.
With these bounds, we have proved the
absolute and uniform convergence of the series \eqref{eq:defining series for F}
on the set $(\partial G\cup\partial(G\setminus\partial G))\times A\times[0,\infty)$.

It is left to prove that the right-hand side of \eqref{eq:Dirichlet_heat_kernel} is
equal to the Dirichlet
heat kernel $H_{D,G}(v_{1},v_{2};t)$ on $G$.  For $v_{1}\in\partial G$, $v_{2}\in G$ and
all $t\geq 0$, and by the definition of $H_{0}$ and
convolution, it is immediate that $H_{D,G}(v_{1},v_{2};t)=0$. What remains is to
prove that $L_{G,v_{1}}H_{0}(v_{1},v_{2};t)=0$ for
all $v_{1}\in VG\setminus\partial G$, $v_{2}\in G$ and all $t\geq0$.

When $v_{1}\in VG\setminus\partial G$, $H_{0}(v_{1},v_{2};t)=\tilde{H}(v_{1},v_{2};t)$.
Hence, by reasoning as in the proof of Lemma \ref{lem: L of convol},
and by using the Dirac property of $\tilde{H}$ we get that
\[
L_{G,v_{1}}(H_{0}\ast F)(v_{1},v_{2};t)=F(v_{1},v_{2};t)+((L_{G,v_{1}}H_{0})\ast F)(v_{1},v_{2};t).
\]
The  absolute and uniform convergence of the series
\eqref{eq:defining series for F} suffices to deduce that $L_{G,v_{1}}H_{0}(v_{1},v_{2};t)=0$
for all $v_{1}\in VG\setminus\partial G$, $v_{2}\in G$ and all $t\geq0$,
as in the proof of Theorem \ref{thm: formula for heat kernel}.
\end{proof}

As an example, let us compute the fundamental solution to the Dirichlet
problem on the discrete half-line $\{0,1,2,\ldots\}$ when starting with
the heat kernel on the graph $\mathbb{Z}$.

\begin{example} \label{ex: half line dirichhlet heat}
With the notation from Example \ref{ex: discrete half line}, let $G$
be the discrete half-line $\{0,1,2,\ldots\}$ when viewed as a subgraph of
$\mathbb{Z}$. Let us set
\[
H_{0}(x,y;t):=\left\{ \begin{array}{ll}
e^{-2t}I_{|x-y|}(2t), & \text{ for }x,y\in\mathbb{Z},x\geq1,y\geq0\\
0, & \text{ if }x=0.
\end{array}\right.
\]
Then the Dirichlet heat kernel on $G$ is given by Proposition \ref{prop: dirichlet heat kernel}.
In order to compute it explicitly, one first notices that
\[
L_{G,x}H_{0}(x,y;t)=\left\{ \begin{array}{ll}
0, & \text{ when }x\geq2;\\
e^{-2t}I_{y}(2t), & \text{ when }x=1\\
-e^{-2t}I_{|y-1|}(2t), & \text{ when }x=0.
\end{array}\right.
\]
Let
\[
F(x,y;t):=\sum_{\ell=1}^{\infty}(-1)^{\ell}(L_{G,x}H_{0})^{\ast\ell}(x,y;t),\quad x,y=0,1,2,...;\quad t\geq0.
\]
When viewed as a function of the first variable $x$, $L_{G,x}H_{0}(x,y;t)$
is supported on the set $\{0,1\}$.  Then, $F(x,y;t)$ as a function of $x$ is
also supported on the set $\{0,1\}$.  Therefore,
we need to find an expression for $F$ when $x=0$ and $x=1$.

\vskip .06in
\textbf{Case 1.} When $x=0$.
First, we compute $(L_{G,x}H_{0})^{\ast\ell}(0,y;t)$
for $\ell\geq2$. When $y\geq1$ we have that
\[
(L_{G,x}H_{0})^{\ast2}(0,y;t)=e^{-2t}\int\limits _{0}^{t}(I_{1}(2(t-\tau))I_{|y-1|}(2\tau)-I_{0}(2(t-\tau))I_{y}(2\tau))d\tau=0
\]
while
$$
(L_{G,x}H_{0})^{\ast2}(0,0;t)=-e^{-2t}I_{1}(2t)=L_{G,x}H_{0}(0,0;t).
$$
Since $(L_{G,x}H_{0})^{\ast2}(0,1;t)=0$, we can now deduce that
$$
(L_{G,x}H_{0})^{\ast2\ell}(0,0;t)=(L_{G,x}H_{0})^{\ast(2\ell-1)}H_{0}(0,0;t)
\,\,\,\,\,
\text{\rm for all $\ell\geq1$.}
$$
It is somewhat straightforward to show that $(L_{G,x}H_{0})^{\ast2\ell}(0,y;t)=0$
unless $y=0$.  Also, we have that
\[
(L_{G,x}H_{0})^{\ast(2\ell+1)}(0,y;t)=\frac{(-1)^{\ell+1}e^{-2t}}{2^{\ell}}((I_{1})^{\ast\ell}\ast I_{(y-1)})(2t),\quad\text{for }y\geq1.
\]
Note that the symbol $\ast$ in the above display is used for two types of convolution: the graph convolution on the left-hand side and the classical convolution of two functions on the right-hand side. We will continue using this notation, as it is clear from the context which type of convolution is used.

With all this, we can conclude that $F(0,0;t)\equiv0$ for all $t\geq0$ and
\[
F(0,y;t)=\sum_{\ell=0}^{\infty}\frac{(-1)^{\ell}e^{-2t}}{2^{\ell}}((I_{1})^{\ast\ell}\ast I_{|y-1|})(2t)
\,\,\,\,\,
\text{\rm for $y\geq 1$,}
\]
where we interpret a $0$-fold convolution to act as the identity operator:
\[((I_l)^{\ast 0}\ast f)(t)=f(t)\mbox{ for all }f:(0,\infty\to\R.
\]

\textbf{Case 2.} $x=1$.  We shall compute $(L_{G,x}H_{0})^{\ast\ell}(1,y;t)$
for $\ell\geq2$. When $\ell=2$ we get that
\begin{align*}
(L_{G,x}H_{0})^{\ast2}(1,y;t) & =e^{-2t}\int\limits _{0}^{t}(I_{1}(2(t-\tau))I_{y}(2\tau)-I_{0}(2(t-\tau))I_{|y-1|}(2\tau))d\tau\\
 & =\left\{ \begin{array}{ll}
0, & \text{ for }y=0;\\
-e^{-2t}I_{y}(2t), & \text{ for }y\geq1.
\end{array}\right.
\end{align*}
Analogous calculations as in Case 1 give that
\[
F(1,y;t)=\left\{ \begin{array}{ll}
-2\sum\limits_{\ell=0}^{\infty}\frac{(-1)^{\ell}e^{-2t}}{2^{\ell}}((I_{1})^{\ast\ell}\ast I_{y})(2t), & \text{ for }y\geq1;\\
-\sum\limits_{\ell=0}^{\infty}\frac{(-1)^{\ell}e^{-2t}}{2^{\ell}}((I_{1})^{\ast\ell}\ast I_{0})(2t)), & \text{ for }y=0.
\end{array}\right.
\]

We now can  compute the convolution of $H_{0}$ and $F$ in closed
form. First, notice that the convolution of $I_{m}$
and $I_{n}$ depends only on the sum of $m$ and $n$, and that it
is commutative operation. Therefore, for $x,y\geq1$ and $t\geq0$
we have
\[
\left(I_{x}\ast \left[(I_{1})^{\ast \ell}\ast I_{y-1}\right]\right)(2t)
=\left(I_{x-1}\ast \left[(I_{1})^{\ast \ell}\ast I_{y}\right]\right)(2t)
=\left[(I_{1})^{\ast \ell}\ast (I_{x-1}\ast I_{y})\right](2t).
\]
The convolution series over
$\ell$ converges absolutely and uniformly  over compact subsets of $VG\times VG\times[0,\infty)$.
With this, we get a closed formula for $H_{0}\ast F(x,y;t)$.  Specifically, for
all $x\geq1$, $y\geq0$ and $t\geq0$, we have that
\[
H_{0}\ast F(x,y;t)=e^{-2t}\sum_{\ell=0}^{\infty}
\frac{(-1)^{\ell+1}e^{-2t}}{2^{\ell+1}}\left((I_{1})^{\ast \ell}\ast (I_{x-1}\ast I_{y})\right)(2t).
\]
Therefore, the Dirichlet heat kernel on the half-line $\{0,1,2,\ldots\}$
is given for $t\geq0$ by
\[
H_{D,G}(x,y;t)=e^{-2t}\left\{ \begin{array}{ll}
I_{|x-y|}(2t)+\sum\limits_{\ell=0}^{\infty}\frac{(-1)^{\ell+1}e^{-2t}}{2^{\ell+1}}
\left((I_{1})^{\ast \ell}\ast (I_{x-1}\ast I_{y})\right)(2t), & \text{ for }x\geq1,y\geq0\\
0, & \text{ if }x=0.
\end{array}\right.
\]

On the other hand, it is easy to check directly that for all $x,y\geq0$ and
$t\geq0$ the Dirichlet heat kernel on the half-line $G$ can be expressed
as
\[
H_{D,G}(x,y;t)=e^{-2t}(I_{x-y}(2t)-I_{x+y}(2t)).
\]
The uniqueness of the Dirichlet heat kernel immediately implies the
identity that
\[
I_{x+y}(t)=\sum_{\ell=0}^{\infty}\frac{(-1)^{\ell}}{2^{\ell+1}}
\left((I_{1})^{\ast \ell}\ast (I_{x-1}\ast I_{y})\right)(t)
\]
for all integers $x\geq1$ and $y\geq0$. From this expression, we will
point out the two interesting special cases that
$$
I_{1}(t)=\sum_{\ell=0}^{\infty}\frac{(-1)^{\ell}}{2^{\ell+1}}\left((I_{1})^{\ast \ell}\ast (I_{0})^{\ast 2}\right)(t)
$$
and
$$
I_{2}(t)=\sum_{\ell=0}^{\infty}\frac{(-1)^{\ell}}{2^{\ell+1}}\left((I_{1})^{\ast (\ell+1)}\ast I_{0}\right)(t).
$$
\end{example}

\vspace{5mm}
\noindent
Gautam Chinta \\
 Department of Mathematics \\
 The City College of New York \\
 Convent Avenue at 138th Street \\
 New York, NY 10031 U.S.A. \\
 e-mail: gchinta@ccny.cuny.edu

\vspace{5mm}
\noindent
Jay Jorgenson \\
 Department of Mathematics \\
 The City College of New York \\
 Convent Avenue at 138th Street \\
 New York, NY 10031 U.S.A. \\
 e-mail: jjorgenson@mindspring.com

\vspace{5mm}
\noindent
Anders Karlsson \\
 Section de mathématiques\\
 Université de Genève\\
 2-4 Rue du Liévre\\
 Case Postale 64, 1211\\
 Genève 4, Suisse\\
 e-mail: anders.karlsson@unige.ch \\
 and \\
 Matematiska institutionen \\
 Uppsala universitet \\
Box 256, 751 05 \\
 Uppsala, Sweden \\
 e-mail: anders.karlsson@math.uu.se

\vspace{5mm}
\noindent
Lejla Smajlovi\'{c} \\
 Department of Mathematics and Computer Science\\
 University of Sarajevo\\
 Zmaja od Bosne 35, 71 000 Sarajevo\\
 Bosnia and Herzegovina\\
 e-mail: lejlas@pmf.unsa.ba
\end{document}